\documentclass[10pt, a4paper]{amsproc}
\usepackage{amsmath,amssymb,amsthm,enumerate,hyperref,lineno,mathtools}

\newtheorem{theorem}{Theorem}[section]
\newtheorem{proposition}[theorem]{Proposition}
\newtheorem{lemma}[theorem]{Lemma}
\newtheorem{corollary}[theorem]{Corollary}
\newtheorem{conjecture}[theorem]{Conjecture}
\newtheorem{question}[theorem]{Question}

\theoremstyle{definition}
\newtheorem{definition}[theorem]{Definition}

\theoremstyle{remark}
\newtheorem{remark}[theorem]{Remark}

\newcommand{\divisible}{\mathrm{div}}
\newcommand{\Kbar}{\overline{K}}
\DeclareMathOperator{\Gal}{Gal}
\DeclareMathOperator{\GL}{GL}
\newcommand{\Gm}{\mathbb{G}_\mathrm{m}}
\newcommand{\pp}{\mathbf{p}}
\newcommand{\qq}{\mathbf{q}}
\newcommand{\rr}{\mathbf{r}}

\newcommand{\tor}{\mathrm{tor}}

\begin{document}

\title[Kummer-faithfulness and the absolute Galois group]{Kummer-faithful fields with finitely generated absolute Galois group}

\author{Takuya Asayama}
\address{Center for Innovative Teaching and Learning, Institute of Science Tokyo, 2-12-1 Ookayama, Meguro-ku, Tokyo 152-8550, Japan}
\email{asayama@citl.isct.ac.jp}


\date{}

\begin{abstract}
This paper studies the structure of the Mordell--Weil groups of semiabelian varieties over algebraic extensions of number fields whose absolute Galois group is finitely generated, with particular emphasis on that generated by a single element.
A probabilistic argument using the Haar measure on the absolute Galois group of a number field shows that almost all such fields are Kummer-faithful, i.e., the Mordell--Weil group of any semiabelian variety over any finite extension of such a field has trivial divisible part.
This result implies that there exists a Kummer-faithful field algebraic over a number field whose absolute Galois group is abelian.
\end{abstract}

\maketitle

\section{Introduction}\label{sec-intro}

This paper focuses on algebraic extensions of a number field whose absolute Galois group is finitely generated.
Fix an algebraic closure $\Kbar$ of any perfect field $K$ and let $G_K$ be the absolute Galois group $\Gal(\Kbar / K)$ of $K$.
Let $e$ be a positive integer.
For any $\sigma = (\sigma_1, \ldots, \sigma_e) \in G_K^e$ (the direct product of $e$ copies of $G_K$), set $\Kbar(\sigma)$ to be the fixed field of $\sigma$ in $\Kbar$, i.e.,
\[ \Kbar(\sigma) = \{x \in \Kbar \mid \sigma_i(x) = x \text{ for all } i\}. \]
Every field with finite transcendence degree over the field $\mathbb{Q}$ of rational numbers and with finitely generated absolute Galois group is of this form.
We equip the compact group $G_K^e$ with the normalized Haar measure $\mu = \mu_{G_K^e}$, which allows $G_K^e$ to be regarded as a probability space~\cite[Section~21.1]{FriedJ}.
The term \textit{almost all} $\sigma \in G_K^e$ is used in the sense of ``all $\sigma \in G_K^e$ outside some measure zero set".

We investigate in this paper how often the field $\Kbar(\sigma)$ is Kummer-faithful.
Here we recall the definitions of Kummer-faithfulness and related notions.

\begin{definition}[{\cite[Definition~1.5]{Mochizuki15}, \cite[Definition~6.1~(iii)]{HMT}}]
A perfect field  $K$  is said to be \textit{Kummer-faithful} (resp.\ \textit{torally Kummer-faithful}, \textit{AVKF} (meaning \textit{abelian variety Kummer-faithful})) if, for every finite extension $L$ of $K$ and every semiabelian variety (resp.\ every torus, every abelian variety) $A$ over $L$, the divisible part $A(L)_\divisible=\bigcap_{n \ge 1} nA(L)$  of the Mordell--Weil group $A(L)$  of  $A$  over  $L$  is trivial; 
\[ A(L)_\divisible = 0. \]
\end{definition}

Notice that the condition $A(L)_\divisible = 0$ in the above definition is equivalent to that the Kummer map associated with $A$ is injective~\cite[p.~2]{Hoshi}.
This property implies that each rational point is uniquely determined by the corresponding Kummer class via the Kummer map, being expected to play an important role in ensuring that the \'{e}tale fundamental group possesses enough information to allow the reconstruction of various geometric objects in the context of anabelian geometry.
Hoshi~\cite{Hoshi} proved some versions of the Grothendieck conjecture in anabelian geometry over Kummer-faithful fields satisfying some additional conditions.

A typical example of Kummer-faithful fields is a sub-$p$-adic field, i.e., a field isomorphic to a subfield of a finitely generated field of the field $\mathbb{Q}_p$ of $p$-adic numbers~\cite[Remark~1.5.4~(i)]{Mochizuki15}, and it would be interesting to find examples of Kummer-faithful fields that are not sub-$p$-adic.

\begin{remark}
It is obvious by definition that any Kummer-faithful field is both torally Kummer-faithful and AVKF.
The converse of each implication is not true; there exist both a torally Kummer-faithful field that is not Kummer-faithful~\cite[Remark~1.5.3~(ii)]{Mochizuki15} and an AVKF field that is not Kummer-faithful (any finite extension of the maximal abelian extension $\mathbb{Q}^\mathrm{ab}$ of $\mathbb{Q}$ is so~\cite[p.~11]{HMT}).
In the characteristic zero case, Kummer-faithfulness is equivalent to the conjunction of torally Kummer-faithfulness and AVKF-ness~\cite[Remark~6.1.3]{HMT}.
\end{remark}

If $K$ is finitely generated over $\mathbb{Q}$, then $\Kbar(\sigma)$ is not sub-$p$-adic for almost all $\sigma \in G_K^e$ and any prime number $p$~\cite[Proposition~5.1]{AT}.
This suggests that it would be worth considering $\Kbar(\sigma)$ for finding Kummer-faithful fields that are not sub-$p$-adic.
Jarden and Petersen~\cite[Theorem~1.3~(ii)]{JP22} proved that any finite extension of $\Kbar(\sigma)$ is AVKF for almost all $\sigma \in G_K^e$ if $K$ is finitely generated over $\mathbb{Q}$ and $e \ge 2$.
The Kummer-faithfulness for a subfield of $\Kbar(\sigma)$ has been studied by Ohtani~\cite{Ohtani22, Ohtani23} and in joint work of the author and Taguchi~\cite{AT}.
Based on these results, we propose the following conjecture.

\begin{conjecture}
Let $K$ be a finitely generated field over $\mathbb{Q}$ and $e$ a positive integer.
Then any finite extension of $\Kbar(\sigma)$ is Kummer-faithful for almost all $\sigma \in G_K^e$.
\end{conjecture}

The main theorem in this paper is that the above conjecture is true if $K$ is a number field.

\begin{theorem}\label{thm-KF}
Suppose that $K$ is a number field and $e$ a positive integer. %finitely generated field over $\mathbb{Q}$.
For almost all $\sigma \in G_K^e$, any finite extension of $\Kbar(\sigma)$ is both torally Kummer-faithful and AVKF, and thus, Kummer-faithful.
\end{theorem}

Notice that the absolute Galois group $G_{\Kbar(\sigma)}$ of $\Kbar(\sigma)$ is abelian when $e = 1$.
This immediately implies that the following corollary.
Actually, more is true: there are continuum many algebraic extensions of $K$ up to $K$-isomorphism which are Kummer-faithful fields with abelian absolute Galois group.
See Corollary~\ref{cor-continuummany} for details.

\begin{corollary}\label{cor-thm-KF}
Let $K$ be a number field.
There exists an algebraic extension of $K$ that is Kummer-faithful and has abelian absolute Galois group.
\end{corollary}

\begin{remark}
One of the assertions in Theorem~1.11 of Mochizuki's paper~\cite{Mochizuki15} claimed that the absolute Galois group of any Kummer-faithful field of characteristic zero is \textit{slim}, i.e., every open subgroup has trivial center.
He recently informed us that the proof of this assertion has a gap.
Corollary~\ref{cor-thm-KF} is incompatible with this assertion and shows the existence of a counterexample in the general case, but the proof is not constructive.
It also answers in the negative the questions~\cite[Remark~2.4.1, Questions~1 and~2]{MT} posed by Minamide and Tsujimura which asked whether the absolute Galois group of any torally Kummer-faithful field is slim.
\end{remark}

\begin{remark}
The assertions related to the (torally) Kummer-faithfulness of $\Kbar(\sigma)$ have appeared twice previously, first in~\cite[Corollary~1]{Ohtani22} and then in the PhD thesis~\cite[Theorem~4.2.1]{AsayamaThesis} of the author, but their proofs in both works were incorrect.
The former has been corrected and the weaker version~\cite[Corollary~1]{Ohtani23} was proved.
In our previous paper~\cite{AT}, the latter was addressed in p.~5 and the weaker result~\cite[Theorem~5.3]{AT} was shown.
\end{remark}

The proof of Theorem~\ref{thm-KF} is carried out by separating it into two parts; the torally Kummer-faithfulness part and the AVKF-ness part.
We provide proofs for each part in Sections~\ref{sec-TKFpart} and~\ref{sec-AVKFpart}, respectively.
The approaches to proving these parts are somewhat similar.
First, we reformulate the problem by reducing it to the calculation of measures.
For the AVKF-ness part, this step has been done by Jarden and Petersen~\cite{JP22}.
Then the calculation of measures is performed.
For the torally Kummer-faithfulness part, this step is based on a refinement of the proof of the weaker version~\cite[Theorem~5.3]{AT} presented in our previous paper.
Finally, we conclude the proof by using a combinatorial approach similar to that described by Zywina~\cite[pp.~495--496]{Zywina16} (see also~\cite[p.~49]{JP19}).
Section~\ref{sec-cor} contains corollaries to Theorem~\ref{thm-KF}.

\section*{Acknowledgments}
The author would like to thank \mbox{Yuichiro} \mbox{Taguchi} and \mbox{Naganori} \mbox{Yamaguchi} for useful discussions and comments throughout this study.

\section{Proof for the torally Kummer-faithfulness}\label{sec-TKFpart}
\subsection{Review of properties of $\Kbar(\sigma)$}

As preliminaries for proving the torally Kummer-faithfulness part of Theorem~\ref{thm-KF}, we make a brief review of a finite extension of $\Kbar(\sigma)$ for $\sigma \in G_K$.
In this subsection, we only assume that the field $K$ is perfect.
We start by showing the following result that appears as an exercise in \textit{Algebra} by Lang~\cite{Lang02}.

\begin{proposition}[{cf.~\cite[Chapter~VI, Exercise~27]{Lang02}}]\label{prop-FinExtOfKbarsigma-new}
Let $K$ be a perfect field and $\sigma \in G_K$.
Any finite extension $M$ over $\Kbar(\sigma)$ is a cyclic extension.
Moreover, its Galois group is generated by $\sigma|_M$.
\end{proposition}

\begin{proof}
Let $M$ be a finite extension over $\Kbar(\sigma)$.
To show that $M / \Kbar(\sigma)$ is cyclic, we may assume that $M$ is Galois by replacing $M$ with its Galois closure over $\Kbar(\sigma)$.
By the definition of $\Kbar(\sigma)$, the fixed subfield of $M / \Kbar(\sigma)$ under $\sigma|_M$ is $\Kbar(\sigma)$.
Galois theory implies that the group generated by $\sigma|_M$ equals $\Gal(M / \Kbar(\sigma))$, which leads to the conclusion of the proposition.
\end{proof}

\begin{lemma}\label{lem-FinExtOfKbarsigma-new}
Let $K$ be a perfect field and $\sigma \in G_K$.
For each positive integer $n$, the field $\Kbar(\sigma)$ has at most one extension of degree $n$.
\end{lemma}

\begin{proof}
Let $M$ and $M'$ be two extensions over $\Kbar(\sigma)$ of the same degree $n$.
Then the composite field $M M'$ is a finite extension over $\Kbar(\sigma)$ and this extension is cyclic by Proposition~\ref{prop-FinExtOfKbarsigma-new}.
Since a cyclic extension has at most one subextension of the same degree, we have $M = M'$.
\end{proof}

\begin{proposition}\label{prop-ClassificationOfFinExtOfKbarsigma}
Let $K$ be a perfect field and $\sigma \in G_K$.
\begin{enumerate}[\textup{(\arabic{enumi})}]
    \item The field $\Kbar(\sigma^n)$ is a finite extension over $\Kbar(\sigma)$ for any positive integer $n$.
    \item Conversely, any finite extension $M$ over $\Kbar(\sigma)$ is of the form $\Kbar(\sigma^n)$ for some positive integer $n$.
\end{enumerate}
\end{proposition}

\begin{proof}
(1) Let $x \in \Kbar(\sigma^n)$.
By Proposition~\ref{prop-FinExtOfKbarsigma-new}, any conjugate of $x$ over $\Kbar(\sigma)$ is of the form $\sigma^i(x)$ for some integer $i$.
Since $\sigma^n(x) = x$, the degree of $\Kbar(\sigma)(x) / \Kbar(\sigma)$ divides $n$.
Hence the extension $\Kbar(\sigma^n) / \Kbar(\sigma)$ is a union of its subextensions of degree at most $n$.
By Lemma~\ref{lem-FinExtOfKbarsigma-new}, there are only finitely many subextensions of degree at most $n$.
This implies that $\Kbar(\sigma^n) / \Kbar(\sigma)$ is finite.

(2) Let $n$ be the degree of $M / \Kbar(\sigma)$.
Since $M / \Kbar(\sigma)$ is a cyclic extension whose Galois group is generated by $\sigma|_M$, we have $M \subset \Kbar(\sigma^n)$.
To show the inverse inclusion, let $x \in \Kbar(\sigma^n)$.
The same argument in the proof of (1) yields that the degree $n'$ of $\Kbar(\sigma)(x) / \Kbar(\sigma)$ divides $n$.
Using again the fact that $M / \Kbar(\sigma)$ is a cyclic extension of degree $n$, we can take a subextension $M'$ of $M / \Kbar(\sigma)$ of degree $n'$.
Then we have $M' = \Kbar(\sigma)(x)$ by Lemma~\ref{lem-FinExtOfKbarsigma-new} and this implies $x \in M$.
\end{proof}

\subsection{Reduction step}

Let $K$ be a finitely generated field over $\mathbb{Q}$.
In order to show Theorem~\ref{thm-KF}, it is obvious that we may assume $e = 1$.
Since the torally Kummer-faithfulness is preserved under taking a finite extension, we only consider the torally Kummer-faithfulness of $\Kbar(\sigma)$ itself.
Then what we have to prove is that the set
\[ S = \{\sigma \in G_K \mid \Kbar(\sigma) \text{ is not torally Kummer-faithful}\} \]
has measure zero in $G_K$.

Our approach to prove this is to express $S$ as a countable union and show that each component has measure zero.
Let us give an example of the expression:
\begin{align*}
    S &= \{\sigma \in G_K \mid \Gm(M)_\divisible \neq 0 \text{ for some finite extension } M / \Kbar(\sigma)\} \\
    &= \{\sigma \in G_K \mid a \in \Gm(\Kbar(\sigma^n))_\divisible \text{ for some } a \in \Kbar^\times \smallsetminus \{1\} \text{ and some } n \ge 1\} \\
    &= \bigcup_{a \in \Kbar^\times \smallsetminus \{1\}} \bigcup_{n \ge 1} S_{a, n},
\end{align*}
where
\[ S_{a, n} = \{\sigma \in G_K \mid a \in \Gm(\Kbar(\sigma^n))_\divisible\}. \]
Note that the second equality follows from Proposition~\ref{prop-ClassificationOfFinExtOfKbarsigma}.
Since the rightmost hand side is a countable union, it is sufficient to show each $S_{a, n}$ has measure zero in $G_K$.
However, there is some difficulty to deal with $S_{a, n}$ for general $a$.
A refinement of this approach enables us to assume that the extension $K(a) / K$ is cyclic in return for varying $K$.

\begin{proposition}\label{prop-TKF-reduction}
Let $K$ be a countable perfect field.
Suppose that, for any finite extension $K' / K$, $a \in \overline{K'}^\times \smallsetminus \{1\}$ with the extension $K'(a) / K'$ cyclic, and a positive integer $n$, the set
\[ \{\sigma \in G_{K'} \mid a \in \Gm(\overline{K'}(\sigma^n))_\divisible\} \]
has measure zero in $G_{K'}$.
Then any finite extension of $\Kbar(\sigma)$ is torally Kummer-faithful for almost all $\sigma \in G_K$.
\end{proposition}

\begin{proof}
By the definition of torally Kummer-faithfulness, we have
\begin{align*}
    S &= \{\sigma \in G_K \mid \Gm(M)_\divisible \neq 0 \text{ for some finite extension } M / \Kbar(\sigma)\} \\
    &= \{\sigma \in G_K \mid \Gm(F \Kbar(\sigma))_\divisible \neq 0 \text{ for some finite Galois extension } F / K\} \\
    &= \bigcup_{\substack{F / K \\ \text{finite Galois}}} \{\sigma \in G_K \mid \Gm(F \Kbar(\sigma))_\divisible \neq 0\} \\
    &= \bigcup_{\substack{F / K \\ \text{finite Galois}}} \bigcup_{a \in \Kbar^\times \smallsetminus \{1\}} \{\sigma \in G_K \mid a \in \Gm(F \Kbar(\sigma))_\divisible\}.
\end{align*}
Taking $F$ larger so that it contains $a$, we may assume $a \in F$, that is, it holds that
\[ S = \bigcup_{\substack{F / K \\ \text{finite Galois}}} \bigcup_{a \in F^\times \smallsetminus \{1\}} \{\sigma \in G_K \mid a \in \Gm(F \Kbar(\sigma))_\divisible\}. \]
Suppose that $\sigma \in G_K$ satisfies $a \in \Gm(F \Kbar(\sigma))_\divisible$ for some finite Galois extension $F / K$ and some $a \in F^\times \smallsetminus \{1\}$.
Set $K' = F \cap \Kbar(\sigma)$.
Then $\sigma$ fixes $K'$, so we have $\sigma \in G_{K'}$ and $\Kbar(\sigma) = \overline{K'}(\sigma)$.
By Proposition~\ref{prop-FinExtOfKbarsigma-new}, the Galois group
\[ \Gal(F / K') \cong \Gal(F \Kbar(\sigma) / \Kbar(\sigma)) \]
is cyclic.
This argument implies the equation
\[ S = \bigcup_{\substack{F / K \\ \text{finite Galois}}} \bigcup_{a \in F^\times \smallsetminus \{1\}} \bigcup_{\substack{K \subset K' \subset F \\ F / K'\text{: cyclic}}} \{\sigma \in G_{K'} \mid a \in \Gm(F \overline{K'}(\sigma))_\divisible\}. \]
Rearranging the order of taking the union and dropping the condition that $F / K$ is Galois yield the inclusion
\[ S \subset \bigcup_{\substack{K' / K \\ \text{finite}}} \bigcup_{\substack{F / K' \\ \text{cyclic}}} \bigcup_{a \in F^\times \smallsetminus \{1\}} \{\sigma \in G_{K'} \mid a \in \Gm(F \overline{K'}(\sigma))_\divisible\}. \]
Since any subextension of a cyclic extension is again cyclic, we may add the condition that $K'(a) / K'$ is cyclic in the union.
By Proposition~\ref{prop-ClassificationOfFinExtOfKbarsigma}, there is a positive integer $n$ such that $F \overline{K'}(\sigma) = \overline{K'}(\sigma^n)$.
Removing the condition $a \in F$, we can eliminate $F$ from the union and obtain
\[ S \subset \bigcup_{\substack{K' / K \\ \text{finite}}} \bigcup_{\substack{a \in \overline{K'}^\times \smallsetminus \{1\} \\ K'(a) / K'\text{: cyclic}}} \bigcup_{n \ge 1} \{\sigma \in G_{K'} \mid a \in \Gm(\overline{K'}(\sigma^n))_\divisible\}. \]
Since $K$ is countable, the right hand side is a countable union.
The assumption of the proposition says that each component in the union has measure zero.
Hence $S$ has also measure zero.
\end{proof}

\subsection{Estimating the measure}

In this subsection, let $K$ be a finitely generated field over $\mathbb{Q}$, $a \in \Kbar^\times \smallsetminus \{1\}$ with the extension $K(a) / K$ cyclic, and $n$ a positive integer.
Part of the discussion in this subsection has been addressed in our previous study~\cite{AT}.
We first consider the case where $a$ is a root of unity.

\begin{proposition}\label{prop-lpowerunity}
Assume that $a \in \Kbar \smallsetminus \{1\}$ is a root of unity.
Then $\mu(S_{a, n}) = 0$ for all positive integer $n$.
\end{proposition}

\begin{proof}
Since $a$ is a root of unity, the set $S_{a, n}$ is contained in the set
\[ \left\{\sigma \in G_K \middle| \begin{array}{l}
    \text{there exist a finite extension $M$ of $\Kbar(\sigma)$ and a prime number} \\
    \text{$\ell$ such that $M^\times$ contains all $\ell$-power roots of unity}
\end{array}\right\}. \]
This set has measure zero in $G_K$ by~\cite[Proposition~5.8]{AT}.
\end{proof}

Next we consider the case where $a$ is not a root of unity.
%Assume that $a$ is not a root of unity.
Let $n$ be a positive integer.
Since $S_{a, n} \subset S_{a, n'}$ for any positive multiple $n'$ of $n$, we may assume that $n$ is a multiple of $r = [K(a) : K]$.
For a prime number $\ell$, we set
\[ T_\ell = \{\sigma \in G_K \mid \text{some $\ell$-th root of $a$ belongs to $\Kbar(\sigma^n)$}\}. \]
Then $S_{a, n} \subset \bigcap_{\ell\text{: prime}} T_\ell$.
Let $\varphi$ be the Euler totient function and $\zeta_m \in \Kbar$ a primitive $m$-th root of unity for a positive integer $m$.

\begin{lemma}[{\cite[Lemma~5.7]{AT}}]\label{lem-mathcalLa}
Let $K$ be a finitely generated field over $\mathbb{Q}$ and $a$ an element in $K^\times$ which is not a root of unity.
Then there exists a set $\Lambda_a$ of prime numbers $\ell \ge 3$ such that:
\begin{itemize}
    \item All but finitely many prime numbers belong to $\Lambda_a$;
    %\item $\sum_{\ell \in \Lambda_a} 1 / \ell = \infty$;
    \item if $n = \ell_1 \cdots \ell_r$, where $\ell_1, \ldots, \ell_r$ are distinct prime numbers in $\Lambda_a$, then the splitting field of $X^n - a$ over $K$ is of degree $n \varphi(n)$ over $K$.
\end{itemize}
\end{lemma}

Let $\Lambda$ be the set of prime numbers $\ell$ satisfying the following conditions:
\begin{itemize}
    \item $\ell$ belongs to $\Lambda_{a'}$ for all conjugates $a'$ of $a$ over $K$, where $\Lambda_{a'}$ is the set of prime numbers obtained from Lemma~\ref{lem-mathcalLa} applying to $K(a)$ and $a'$ (Note that $a' \in K(a)$ since the extension $K(a) / K$ is Galois);
    \item $\ell$ is congruent to $1$ modulo $r$;
    \item $\ell$ does not divide $n$.
\end{itemize}
Then $\Lambda$ contains all but finitely many prime numbers congruent to $1$ modulo $r$ and thus has positive Dirichlet density~\cite[Corollary~7.3.2]{FriedJ}.

Let $\ell \in \Lambda$.
Since $[K(a, \zeta_\ell) : K(a)] = \ell - 1$, we have
\[ [K(a, \zeta_\ell) : K] = [K(a, \zeta_\ell) : K(a)] [K(a) : K] = [K(\zeta_\ell) : K] [K(a) : K]. \]
Hence the extensions $K(\zeta_\ell)$ and $K(a)$ are linearly disjoint over $K$ and we can take $\tau \in G_K$ satisfying:
\begin{itemize}
    \item $\tau$ fixes $\zeta_\ell$;
    \item the restriction of $\tau$ to $K(a)$ is a generator of the cyclic group $\Gal(K(a) / K)$.
\end{itemize}
By the second condition, we have
\[ G_K = \bigsqcup_{j = 0}^{r - 1} G_{K(a)} \tau^j. \]

Fix $0 \le j \le r - 1$.
Then the field $K' = K(a) \cap \Kbar(\sigma)$ does not depend on $\sigma \in G_{K(a)} \tau^j$ and the restriction of $\tau^j$ to $K(a)$ is a generator of the cyclic group $\Gal(K(a) / K')$.
Let $a_1 = a, a_2, \ldots, a_s$ be all conjugates of $a$ over $K'$ indexed so that
\[ a_1 \overset{\tau^j}{\longmapsto} a_2 \overset{\tau^j}{\longmapsto} \cdots \overset{\tau^j}{\longmapsto} a_s \overset{\tau^j}{\longmapsto} a_1. \]
Every $\sigma \in G_{K'}$ induces a permutation of the set $\{1, 2, \ldots, s\}$ of indices via the action on the set $\{a_1, a_2, \ldots, a_s\}$.
Abusing notation, we also write $\sigma$ for this permutation as an element in the symmetric group $\mathfrak{S}_s$ of $s$ letters.
Then $\sigma$ belongs to the subgroup $\langle(1 \, 2 \, \cdots \, s)\rangle \subset \mathfrak{S}_s$.
Note that the permutation $(1 \, 2 \, \cdots \, s)$ is the one to which $\tau^j$ corresponds.
Fix an $\ell$-th root $\alpha_i$ of $a_i$ for each $1 \le i \le s$.
We define a representation $\rho_{\ell, j}: G_{K'} \to \GL_{s + 1}(\mathbb{F}_\ell)$ as
\[ \rho_{\ell, j}(\sigma) = \left(\begin{array}{c|ccc}
    u(\sigma) & b_1(\sigma) & \cdots & b_s(\sigma) \\ \hline
     & \multicolumn{3}{c}{P(\sigma)}
\end{array}
\right),
\]
where $u(\sigma) \in \mathbb{F}_\ell^\times$, $b_i(\sigma) \in \mathbb{F}_\ell$ ($1 \le i \le s$), and $P(\sigma) \in \GL_s(\mathbb{F}_\ell)$ are defined as follows:
\begin{itemize}
    \item $\sigma(\zeta_\ell) = \zeta_\ell^{u(\sigma)}$,
    \item $\sigma(\alpha_i) = \zeta_\ell^{b_i(\sigma)} \alpha_{\sigma(i)}$,
    \item $P(\sigma)$ is a permutation matrix corresponding to $\sigma$; e.g.,
    \[ P(\tau^j) = \begin{pmatrix}
        & & & 1 \\
        1 & & & \\
        & \ddots & & \\
        & & 1 &
    \end{pmatrix}.
    \]
\end{itemize}
Since $n$ is a multiple of $s$, the matrix $P(\sigma^n)$ is always identity and $\sigma^n \in G_{K(a)}$.
By definition, for $\sigma \in G_K$ with $u(\sigma^n) = 1$, we have $\sigma \in T_\ell$ if and only if $b_1(\sigma^n) = 0$.

The image of any element in $G_{K(a)} \tau^j$ is of the form
\[ \left(\begin{array}{c|cccc}
    u & b_1 & \cdots & b_{s - 1} & b_s \\ \hline
     & & & & 1 \\
     & 1 & & & \\
     & & \ddots & & \\
     & & & 1 &
\end{array}
\right)
\]
for $u \in \mathbb{F}_\ell^\times$ and $b_i \in \mathbb{F}_\ell$ ($1 \le i \le s$), and the composition of maps
\[ G_{K(a)} \tau^j \xrightarrow{\rho_{\ell, j}} \rho_{\ell, j}(G_{K(a)} \tau^j) \xrightarrow{{\scriptsize \begin{array}{c}
    \text{extracting the} \\
    \text{upper left entry}
\end{array}}} \mathbb{F}_\ell^\times \]
is surjective.

Since $[K(a, \zeta_\ell) : K(a)] = \ell - 1$, we have the decomposition into cosets
\[ G_{K(a)} = \bigsqcup_{u \in \mathbb{F}_\ell^\times} G_{K(a, \zeta_\ell)} \eta_u, \]
where $\eta_u \in G_{K(a)}$ satisfies $\eta_u(\zeta_\ell) = \zeta_\ell^u$.
Set $F = K(\sqrt[\ell]{a_1}, \ldots, \sqrt[\ell]{a_s})$.
By Kummer theory, the Galois group $\Gal(F / K(a, \zeta_\ell))$ is isomorphic to $\mathbb{F}_\ell^{s'}$ for $1 \le s' \le s$.
Via the natural surjection from $\mathbb{F}_\ell^{s'}$ to the projective space $\mathbb{P}_{\mathbb{F}_\ell}^{s' - 1}$, we define the equivalence relation $\sim$ on $\Gal(F / K(a, \zeta_\ell)) \smallsetminus \{1\} = \left(G_F \middle\backslash G_{K(a, \zeta_\ell)}\right) \smallsetminus \{1\}$.
Then we have
\[ G_{K(a, \zeta_\ell)} = \left(\bigsqcup_{\xi \in X_{\ell, j}} \bigsqcup_{h = 1}^{\ell - 1} G_F \xi^h \right) \sqcup G_F, \]
where $X_{\ell, j} = \left.\left(\left(G_F \middle\backslash G_{K(a, \zeta_\ell)}\right) \smallsetminus \{1\} \right) \middle/ {\sim}\right.$.
Choosing a transversal of $X_{\ell, j}$, we view it as a subset in $\left(G_F \middle\backslash G_{K(a, \zeta_\ell)}\right) \smallsetminus \{1\}$.
Therefore we obtain
\begin{align*}
    G_{K(a)} \tau^j &= \bigsqcup_{u \in \mathbb{F}_\ell^\times} \left(\left(\bigsqcup_{\xi \in X_{\ell, j}} \bigsqcup_{h = 1}^{\ell - 1} G_F \xi^h \right) \sqcup G_F\right) \eta_u \tau^j \\
    &= \left(\bigsqcup_{u \in \mathbb{F}_\ell^\times} \bigsqcup_{\xi \in X_{\ell, j}} \bigsqcup_{h = 1}^{\ell - 1} G_F \xi^h \eta_u \tau^j\right) \sqcup \left(\bigsqcup_{u \in \mathbb{F}_\ell^\times} G_F \eta_u \tau^j\right).
\end{align*}
Namely, there are $\ell^{s'} (\ell - 1)$ cosets in $G_F \backslash G_{K'}$ contained in $G_{K(a)} \tau^j$ and each of them is uniquely of the form $G_F \xi^h \eta_u \tau^j$ for $u \in \mathbb{F}_\ell^\times$, $\xi \in X_{\ell, j}$, and $1 \le h \le \ell - 1$ or of the form $G_F \eta_u \tau^j$ for $u \in \mathbb{F}_\ell^\times$.

Let $\xi \in X_{\ell, j}$.
Since $\xi \notin G_F$, the image $\rho_{\ell, j}(\xi)$ is of the form
\[ \begin{pmatrix}
    1 & t_1 & \cdots & t_s \\
     & 1 & & \\
     & & \ddots & \\
     & & & 1
\end{pmatrix}
\]
with $(t_1 \, \cdots \, t_s) \neq 0$.
Since the $s \times s$ Vandermonde matrix consisting of geometric progressions with common ratio $u$ for distinct $s$-th roots $u$ of unity in $\mathbb{F}_\ell$ has nonzero determinant, there exists $u_0 \in \mathbb{F}_\ell$ such that $u_0^s = 1$ and 
\[ u_0^{s - 1} t_1 + u_0^{s - 2} t_2 + \cdots + t_s \neq 0. \]
Note that, as $s$ divides $r$ and $r$ divides $\ell - 1$, all $s$-th roots of unity belong to $\mathbb{F}_\ell$.
Recalling that $\tau^j$ fixes $\zeta_\ell$, we write
\[ \rho_{\ell, j}(\eta_{u_0} \tau^j) = \left(\begin{array}{c|cccc}
    u_0 & b_1 & \cdots & b_{s - 1} & b_s \\ \hline
     & & & & 1 \\
     & 1 & & & \\
     & & \ddots & & \\
     & & & 1 &
\end{array}
\right). \]
Then we have
\[ \rho_{\ell, j}(\xi^h \eta_{u_0} \tau^j) = \left(\begin{array}{c|ccccc}
    u_0 & b_1 + h t_2 & b_2 + h t_3 & \cdots & b_{s - 1} + h t_s & b_s + h t_1 \\ \hline
     & & & & & 1 \\
     & 1 & & & & \\
     & & 1 & & &\\
     & & & \ddots & & \\
     & & & & 1 &
\end{array}
\right) \]
for any integer $h$.
Raising both sides to the $s$-th power gives
\[ \rho_{\ell, j}((\xi^h \eta_{u_0} \tau^j)^s) = \begin{pmatrix}
    1 & v_h & \cdots \\
     & 1 & \\
     & & \ddots
\end{pmatrix}
\]
with
\begin{align*}
    v_h &= u_0^{s - 1} (b_1 + h t_2) + u_0^{s - 2} (b_2 + h t_3) + \cdots + u_0 (b_{s - 1} + h t_s) + (b_s + h t_1) \\
    &= (u_0^{s - 1} b_1 + u_0^{s - 2} b_2 + \cdots + b_s) + h u_0 (u_0^{s - 1} t_1 + u_0^{s - 2} t_2 + \cdots + t_s).
\end{align*}
Since $u_0 (u_0^{s - 1} t_1 + u_0^{s - 2} t_2 + \cdots + t_s) \neq 0$, this $v_h$ varies over all elements in $\mathbb{F}_\ell$ when $h$ varies over the range $0 \le h \le \ell - 1$.
Removing $h = 0$, we have at least $\ell - 2$ values of $h$ in the range $1 \le h \le \ell - 1$ with $v_h \neq 0$.
For such $h$, we have
\[ \rho_{\ell, j}((\xi^h \eta_{u_0} \tau^j)^n) = \begin{pmatrix}
    1 & \dfrac{n}{s} v_h & \cdots \\
     & 1 & \\
     & & \ddots
\end{pmatrix}.
\]
Since $n$ is not divisible by $\ell$, we have $(n / s) v_h \neq 0$.
This means that any element in the coset $G_F \xi^h \eta_{u_0} \tau^j$ is not a member of $T_\ell$.

To summarize, for each $\xi \in X_{\ell, j}$, we obtain at least $\ell - 2$ cosets of $G_F \backslash G_{K'}$ contained in $G_{K(a)} \tau^j$ having empty intersection with $T_\ell$.
By the uniqueness of the expression of each coset of the form $G_F \xi^h \eta_u \tau^j$, there are at least $(\ell - 2) (\# X_{\ell, j})$ cosets contained in $G_{K(a)} \tau^j$ with this property.
Since $\# X_{\ell, j} = (\ell^{s'} - 1) / (\ell - 1)$, we have
\[ \frac{\mu(T_\ell \cap G_{K(a)} \tau^j)}{\mu(G_{K(a)} \tau^j)} \le 1 - \frac{(\ell - 2) (\# X_{\ell, j})}{\ell^{s'} (\ell - 1)} = 1 - \frac{\ell - 2}{(\ell - 1)^2} \frac{\ell^{s'} - 1}{\ell^{s'}} \le 1 - \frac{1}{2 \ell}. \]
Here we use $\ell \ge 3$ and $s' \ge 1$ in the last inequality.
As a consequence, we obtain the following result.

\begin{proposition}\label{prop-measureofTellcapcoset}
We have
\[ \mu(T_\ell \cap G_{K(a)} \tau) \le \frac{1}{r} \left(1 - \frac{1}{2 \ell}\right) \]
for any $\ell \in \Lambda$ and any $\tau \in G_K$.
\end{proposition}

\subsection{Independence and the completion of the proof}

In this subsection, we discuss the independence of the subsets $T_\ell$ and complete the proof for the torally Kummer-faithfulness part of Theorem~\ref{thm-KF}.
We use the following result, which is a special case of the theorem proved by Perucca, Sgobba, and Tronto~\cite{PST}.

\begin{theorem}[Perucca--Sgobba--Tronto~{\cite[Theorem~1.1]{PST}}]\label{PST}
Let $K$ be a number field and $a_1, \ldots, a_s \in K^\times$.
Suppose that the subgroup in $\Gm(K)$ generated by $a_1, \ldots, a_s$ is isomorphic to $\mathbb{Z}^s$.
Then there exists a positive integer $N$ depending only on $K, a_1, \ldots, a_s$ such that
\[ [K(\sqrt[m]{a_1}, \ldots, \sqrt[m]{a_s}) : K] = m^s \varphi(m) \]
for any positive integer $m$ coprime to $N$.
In particular, the family of field extensions $\{K(\sqrt[\ell]{a_1}, \ldots, \sqrt[\ell]{a_s})\}_{\ell \nmid N}$ is linearly disjoint over $K$.
\end{theorem}

\begin{proof}[Proof of Theorem~\textup{\ref{thm-KF}} (for the torally Kummer-faithfulness)]
It is sufficient to prove that the assumption of Proposition~\ref{prop-TKF-reduction} holds.
Let $K$ be a number field and $a \in \Kbar^\times \smallsetminus \{1\}$ with $K(a) / K$ cyclic.
What we need to show is that $\mu(S_{a, n}) = 0$ for any positive integer $n$.

The case where $a$ is a root of unity has been shown in Proposition~\ref{prop-lpowerunity}.
Suppose that $a$ is not a root of unity.
Let $H_a$ be the subgroup of $\Gm(K(a))$ generated by all conjugates of $a$ over $K$.
For a prime number $\ell$, we write $K(\sqrt[\ell]{H_a})$ for the extension over $K$ obtained by adjoining all $\ell$-th roots of all elements in $H_a$.

Suppose that $H_a$ has nontrivial torsion.
Since the torsion is finite, we can take a sufficiently large power of $a$, say $a^w$, so that the subgroup of $\Gm(K(a))$ generated by conjugates of $a^w$ over $K$ has trivial torsion.
This enables us to assume without loss of generality that $H_a$ has trivial torsion since $S_{a, n} \subset S_{a^w, n}$.

Let $\nu = r \mu$, where $r = [K(a) : K]$.
Then we may view the function $\nu$ as the normalized Haar measure on $G_{K(a)}$.
Let $\tau \in G_K$ and $\Lambda$ be the set of prime numbers $\ell$ satisfying the three conditions described after Lemma~\ref{lem-mathcalLa}.
For $\ell \in \Lambda$, whether $\sigma \in G_K$ belongs to $T_\ell$ is determined only by its restriction to $K(\sqrt[\ell]{H_a})$.
Hence the subset $(T_\ell \cap G_{K(a)} \tau) \tau^{- 1} \subset G_{K(a)}$ is a union of cosets of $G_{K(\sqrt[\ell]{H_a})} \backslash G_{K(a)}$.
By Theorem~\ref{PST}, the family of field extensions $\{K(\sqrt[\ell]{H_a})\}_{\ell \in \Lambda'}$ is linearly disjoint over $K(a)$ for some set $\Lambda'$ obtained by removing a finite number of prime numbers from $\Lambda$.
Then the family of measurable sets $\{(T_\ell \cap G_{K(a)} \tau) \tau^{- 1}\}_{\ell \in \Lambda'}$ is $\nu$-independent in $G_{K(a)}$~\cite[Lemma~21.3.7]{FriedJ}.
Combining Proposition~\ref{prop-measureofTellcapcoset} and the fact that $\Lambda'$ has positive Dirichlet density, we have
\begin{align*}
    \mu(S_{a, n} \cap G_{K(a)} \tau) &\le \mu\left(\bigcap_{\ell \in \Lambda'} T_\ell \cap G_{K(a)} \tau\right) = \mu\left(\bigcap_{\ell \in \Lambda'} (T_\ell \cap G_{K(a)} \tau) \tau^{- 1}\right) \\
    &= \frac{1}{r} \nu\left(\bigcap_{\ell \in \Lambda'} (T_\ell \cap G_{K(a)} \tau) \tau^{- 1}\right) = \frac{1}{r} \prod_{\ell \in \Lambda'} \nu((T_\ell \cap G_{K(a)} \tau) \tau^{- 1}) \\
    &\le \frac{1}{r} \prod_{\ell \in \Lambda'} \left(1 - \frac{1}{2 \ell}\right) = 0.
\end{align*}
Running $\tau$ and summing up over all cosets of $G_{K(a)} \backslash G_K$, we obtain $\mu(S_{a, n}) = 0$.
\end{proof}

\section{Proof for the AVKF-ness}\label{sec-AVKFpart}

This section studies the AVKF-ness of finite extensions of $\Kbar(\sigma)$.
As already mentioned in Introduction, the AVKF-ness part for $e \ge 2$ of Theorem~\ref{thm-KF} is proved by Jarden and Petersen~\cite{JP22}.
Even in the $e = 1$ case, much of their argument continues to hold.
Actually, it turns out that all we have to show is one claim because this is the only point where they used the assumption $e \ge 2$.
Let us start with describing what we need to show.
In what follows, let $K$ be a finitely generated field of $\mathbb{Q}$ (later, we will assume that $K$ is a number field).

\subsection{Overview of Jarden--Petersen's strategy}\label{sec-JP}

We first remark that it is of our interest how often $\Kbar(\sigma)$ itself is AVKF since AVKF-ness is preserved under finite extensions.
Moreover, using Weil's restriction for abelian varieties~\cite[Lemma~6.1]{JP22}, the AVKF-ness part of Theorem~\ref{thm-KF} is reduced to showing that the group $A(\Kbar(\sigma))$ has no nontrivial divisible point for almost all $\sigma \in G_K^e$ and all abelian varieties $A$ over $\Kbar(\sigma)$.

In Section~5 of their paper~\cite{JP22}, Jarden and Petersen proved that this claim is true for $e \ge 2$.
Lemma~5.1, Lemma~5.3, and Proposition~5.4 in~\cite{JP22} show without assuming $e \ge 2$ that this claim is reduced to that the following two statements hold for any finite extension $K'$ over $K$ and any abelian variety $A$ over $K'$:
\begin{itemize}
    \item $A(K') \cap A(\overline{K'}(\sigma))_\divisible \subset A(K')_\tor$ for almost all $\sigma \in G_{K'}^e$;
    \item $A(\overline{K'}(\sigma))[\ell^\infty]$ is finite for almost all $\sigma \in G_{K'}^e$ and all prime numbers $\ell$.
\end{itemize}
As the latter statement is already known to hold~\cite[Main Theorem]{JJ01}, we need to prove the former.
Jarden and Petersen proved a slightly stronger result.

\begin{lemma}[{\cite[Lemma~5.2]{JP22}}]\label{lem-JP-key}
Assume $e \ge 2$.
Let $A$ be an abelian variety over $K$.
Then $A(K) \cap \left(\bigcap_{\ell\text{\textup{{: prime}}}} \ell A(\Kbar(\sigma))\right) \subset A(K)_\tor$ for almost all $\sigma \in G_K^e$.
\end{lemma}

In their proof of this lemma, Jarden and Petersen reduced it to the case where $A$ is simple.
In what follows, we assume that $A$ is a simple abelian variety over $K$ and let $g$ be its dimension.
Since $A(K)$ is at most countable, it is sufficient to show that $\pp \notin \bigcap_{\ell\text{: prime}} \ell A(\Kbar(\sigma))$ for all non-torsion point $\pp \in A(K)$ and almost all $\sigma \in G_K^e$.
Namely, we have to prove that the set
\[ N(\pp) = \left\{\sigma \in G_K^e \,\middle|\, \pp \in \bigcap_{\ell\text{: prime}} \ell A(\Kbar(\sigma)) \right\} \]
has measure zero in $G_K^e$ for any non-torsion point $\pp \in A(K)$.

For a prime number $\ell$, set
\begin{align*}
    X_\ell(\pp) = \{ \qq \in A(\Kbar) \mid \ell \qq = \pp \}
\end{align*}
to be the set of $\ell$-division points of $\pp$.
The key point is the following result.

\begin{proposition}[{\cite[Corollary~4.4]{JP22}}]\label{prop-JP-key}
Let $A$ be a simple abelian variety over $K$ of dimension $g$ and $\pp \in A(K)$ a non-torsion point.
Then, for all sufficiently large prime numbers $\ell$, we have
\[\Gal(K(X_\ell(\pp)) / K(A[\ell])) \cong A[\ell]. \]
In particular, we have $[K(\qq) : K] = \ell^{2 g}$ for any $\qq \in X_\ell(\pp)$.
\end{proposition}

We have
\begin{align*}
    N(\pp) &= \bigcap_{\ell\text{: prime}} \{\sigma \in G_K^e \mid \pp \in \ell A(\Kbar(\sigma))\} \\
    &= \bigcap_{\ell\text{: prime}} \{\sigma \in G_K^e \mid X_\ell(\pp) \cap A(\Kbar(\sigma)) \neq \varnothing\}.% \\
\end{align*}
For a prime number $\ell$, set
\[ S_\ell(\pp) = \{ \sigma \in G_K^e \mid X_\ell(\pp) \cap A(\Kbar(\sigma)) \neq \varnothing\} \]
to be the set of $\sigma \in G_K^e$ such that $A(\Kbar(\sigma))$ contains at least one $\ell$-division point of $\pp$.
Then we see
\[ S_\ell(\pp) = \bigcup_{\qq \in X_\ell(\pp)} \{\sigma \in G_K^e \mid \qq \in A(\Kbar(\sigma))\} = \bigcup_{\qq \in X_\ell(\pp)} G_{K(\qq)}^e. \]
By Proposition~\ref{prop-JP-key} and the fact that $\# X_\ell(\pp) = \# A[\ell] = \ell^{2 g}$, we estimate
\[ \mu(S_\ell(\pp)) \le \sum_{\qq \in X_\ell(\pp)} \mu(G_{K(\qq)}^e) = \sum_{\qq \in X_\ell(\pp)} \frac{1}{[K(\qq) : K]^e} = \frac{\ell^{2 g}}{(\ell^{2 g})^e} = \frac{1}{\ell^{2 g (e - 1)}} \]
for all sufficiently large $\ell$.
If we assume that $e \ge 2$, then this implies that $\mu(S_\ell(\pp))$ tends to zero as $\ell$ goes to infinity.
Hence we conclude that $\mu(N(\pp)) = 0$, which yields Lemma~\ref{lem-JP-key}.

If $e = 1$, then the rightmost hand in the above estimation is always one, which prevents us from extracting any information on the measures of $S_\ell(\pp)$ and of $N_\ell(\pp)$.
This failure is due to the rough estimation in the inequality
\[ \mu(S_\ell(\pp)) \left(= \mu\left(\bigcup_{\qq \in X_\ell(\pp)} G_{K(\qq)}^e\right)\right) \le \sum_{\qq \in X_\ell(\pp)} \mu(G_{K(\qq)}^e). \]
Although the sum of the measures of $G_{K(\qq)}$ for $\qq \in X_\ell(\pp)$ equals one, their intersection obviously has positive measure and the measure of $S_\ell(\pp)$ is strictly less than one.
A more precise estimate can be obtained, and when $K$ is a number field, a result concerning the independence of the family $\{S_\ell(\pp)\}_\ell$ can be established, thereby allowing us to extend Lemma~\ref{lem-JP-key} to the case $e = 1$.

\begin{lemma}\label{AVKF-Goal}
Let $K$ be a number field and $A$ a simple abelian variety over $K$.
For any non-torsion point $\mathbf{p}$ in $A(K)$, the set
\[ N(\pp) = \left\{\sigma \in G_K \,\middle|\, \pp \in \bigcap_{\ell\text{\textup{: prime}}} \ell A(\Kbar(\sigma)) \right\} \]
has measure zero in $G_K$.
Therefore, if $K$ is a number field, then Lemma~\textup{\ref{lem-JP-key}} is true even for $e = 1$, i.e., $A(K) \cap \left(\bigcap_{\ell\text{\textup{: prime}}} \ell A(\Kbar(\sigma))\right) \subset A(K)_\tor$ for almost all $\sigma \in G_K$ and any abelian variety $A$ over $K$.
\end{lemma}

The rest of this section is devoted to the proof of Lemma~\ref{AVKF-Goal}.
We first examine the independence of $\{S_\ell(\pp)\}_\ell$ and the measure of each $S_\ell(\pp)$ in separate subsections, and then conclude in the final subsection.

\subsection{Independence}\label{sec-independence}

In this subsection, let $K$ be a number field, $A$ an abelian variety over $K$ of dimension $g \ge 1$, and $\pp \in A(K)$ a non-torsion point.

Let
\[ a_\ell: G_K \to \mathrm{Aut}(A[\ell]) \cong \mathrm{GL}_{2 g}(\mathbb{F}_\ell) \]
be the Galois representation on $A[\ell]$.
Choose $\qq \in X_\ell(\pp)$ and let
\[ b_\ell(\sigma) = \sigma(\qq) - \qq \in A[\ell] \]
for $\sigma \in G_K$.
Then
\[ b_\ell: G_K \to A[\ell] \cong \mathbb{F}_\ell^{2 g} \]
is a $1$-cocycle.
Hence $a_\ell$ and $b_\ell$ induce the representation
\[ \rho_\ell: G_K \to \mathrm{GL}_{2 g + 1}(\mathbb{F}_\ell) \]
defined as
\[ \rho_\ell(\sigma) = \begin{pmatrix}
    a_\ell(\sigma) & b_\ell(\sigma) \\
    0 & 1
\end{pmatrix}.
\]

\begin{proposition}\label{prop-independence}
Let $K$ be a number field and $A$ an abelian variety over $K$.
Then there exists a finite extension $L$ of $K$ such that the homomorphism
\[ G_L \to \prod_{\ell} \rho_\ell(G_L) \]
induced by $\{\rho_\ell\}_\ell$ is surjective.
\end{proposition}

\begin{proof}
We use the theorem~\cite[Th\'{e}or\`{e}me~1]{Serre13} of Serre on the independence of a family of $\ell$-adic Galois representations.
Notice that a result similar to what we are now proving holds for the family $\{a_\ell\}_\ell$~\cite[3.1]{Serre13}.
Then there exists a finite extension $L$ of $K$ such that
\[ G_L \to \prod_{\ell} a_\ell(G_L) \]
induced by $\{a_\ell\}_\ell$ is surjective.
We also know that $A$ has good reduction at all but finitely many finite places of $L$.
Then what we need to prove is reduced to showing that, for any finite place $v$ where $A$ has good reduction and any prime number $\ell$ different from the characteristic $p_v$ of the residue field of $v$, the representation $\rho_\ell$ is unramified at $v$.
Since $A(L_v)$ is isomorphic to the direct sum of a finitely generated free $\mathbb{Z}_{p_v}$-module and finite torsion~\cite[Theorem~7]{Mattuck} and since $\ell \neq p_v$, the multiplication-by-$\ell$ map on $A(L_v) / A(L_v)_\tor$ is bijective.
Hence we can take $\rr \in A(L_v)$ such that $\ell \rr - \pp$ is a torsion point.
Then any $\ell$-division point of $\ell \rr - \pp$, which is of the form $\rr - \qq$ for $\qq \in X_\ell(\pp)$, is defined over $L_v(A[\ell^m])$ for some $m \ge 1$.
Since $L_v(A[\ell^m])$ is unramified at $v$ by the criterion of N\'{e}ron--Ogg--Shafarevich, so is $L_v(X_\ell(\pp))$.
This means that $\rho_\ell$ is unramified at $v$.
\end{proof}

\subsection{Computing of the measure of $S_\ell(\pp)$}

We keep the same notation as in the previous subsection; $K$ is a number field, $A$ is an abelian variety over $K$ of dimension $g \ge 1$, and $\pp$ is a non-torsion point in $A(K)$.
We additionally assume that $A$ is simple.
In this subsection, we estimate the measure of the set $S_{\ell}(\pp)$.
To allow us to apply Proposition~\ref{prop-independence}, we need to show that --- $S_{\ell}(\pp)$ uniformly distributes --- the measures of $S_{\ell}(\pp) \cap \beta H$, where $H$ is an open subgroup of $G_K$ and $\beta \in G_K$, have a uniform estimate independent of the choice of $\beta \in G_K$.
Notice that all results in this subsection work for a general finitely generated field $K$ of $\mathbb{Q}$. 

For any prime number $\ell$, the set $A[\ell]$ of $\ell$-torsion points in $A(\Kbar)$ forms a $2 g$-dimensional $\mathbb{F}_\ell$-vector space and the subset $W(\sigma) = A[\ell] \cap A(\Kbar(\sigma))$ forms an $\mathbb{F}_\ell$-vector subspace of $A[\ell]$ for any $\sigma \in G_K$.
Let $L$ be a finite extension of $K$ and $\beta \in G_K$.
For $0 \le d \le 2 g$, set
\[ Y_d = \{\sigma \in \beta G_L \mid \dim_{\mathbb{F}_\ell} W(\sigma) \ge d\}. \]
Note that $\mu(Y_0) = \mu(\beta G_L) = 1 / [L : K]$.
The goal of this subsection is to prove the following result.

\begin{proposition}\label{prop-ComputationOfTheMeasure}
Let $L$ be a finite extension of $K$ and $\beta \in G_K$.
For all sufficiently large prime numbers $\ell$, we have
\[ \mu(S_\ell(\pp) \cap \beta G_L) = \frac{1}{[L : K]} - \sum_{d = 1}^{2 g} \frac{\ell - 1}{\ell^d} \mu(Y_d). \]
\end{proposition}

To prove this proposition, we fix a finite extension $L$ of $K$ and $\beta \in G_K$.
For a prime number $\ell$ and $0 \le d \le 2 g$, set $\mathcal{W}_d$ to be the set of all $\mathbb{F}_\ell$-vector subspaces of dimension $d$ in $A[\ell]$ and
\[ \mathcal{W} = \bigcup_{d = 0}^{2 g} \mathcal{W}_d \]
to be the set of all $\mathbb{F}_\ell$-vector subspaces in $A[\ell]$.
Then we can write
\[ Y_d = \bigcup_{W \in \mathcal{W}_d} (G_{K(W)} \cap \beta G_L). \]
The following lemma is easy to check.

\begin{lemma}
Suppose $\sigma \in S_\ell(\pp)$ and choose $\qq_0 \in X_\ell(\pp) \cap A(\Kbar(\sigma))$.
Then
\[ W(\sigma) = \{\qq - \qq_0 \mid \qq \in X_\ell(\pp) \cap A(\Kbar(\sigma))\}. \]
In particular, the right hand side is independent of the choice of $\qq_0$.
\end{lemma}

For $W \in \mathcal{W}$, define a binary relation $\sim_W$ on $X_\ell(\pp)$ as follows:
\[ \qq \sim_W \qq' \Longleftrightarrow \qq - \qq' \in W. \]
It is easily seen that $\sim_W$ is an equivalent relation.
For $\qq \in X_\ell(\pp)$ and $W \in \mathcal{W}$, set
\[ J(\qq, W) = \{\sigma \in \beta G_L \mid \qq \in A(\Kbar(\sigma)) \text{ and } W(\sigma) = W\}. \]
It immediately follows that $\qq \sim_W \qq'$ implies $J(\qq, W) = J(\qq', W)$.

\begin{lemma}\label{PartitionByWq}
The family $\{J(\qq, W)\}_{W \in \mathcal{W}, \qq \in X_\ell(\pp) / {\sim_W}}$ gives a partition of the set $S_\ell(\pp) \cap \beta G_L$.
Here we admit the case where $J(\qq, W)$ is empty for some $\qq$ and $W$.
\end{lemma}

\begin{proof}
By the definition of $J(\qq, W)$, the family
\[ \left\{\bigcup_{\qq \in X_\ell(\pp) / {\sim_W}} J(\qq, W)\right\}_{W \in \mathcal{W}} \]
gives a partition of the set $S_\ell(\pp) \cap \beta G_L$.
Then it suffices to show that the union $\bigcup_{\qq \in X_\ell(\pp) / {\sim_W}} J(\qq, W)$ is disjoint for any $W \in \mathcal{W}$.
Let $W \in \mathcal{W}$ and $\qq, \qq' \in X_\ell(\pp)$, and suppose that $J(\qq, W) \cap J(\qq', W) \neq \varnothing$.
Take any $\sigma$ from the left hand side.
Then we have $\qq, \qq' \in X_\ell(\pp) \cap A(\Kbar(\sigma))$ and $\qq - \qq' \in W(\sigma) = W$.
This implies $\qq \sim_W \qq'$, as desired.
\end{proof}

For $\qq \in X_\ell(\pp)$, $W \in \mathcal{W}$, and $0 \le d \le 2 g$, we define the following measures:
\begin{align*}
    Z_{\qq, W} &= \mu(J(\qq, W)), \\
    T_{\qq, d} &= \sum_{W \in \mathcal{W}_d} Z_{\qq, W} \\
    &= \mu(\{\sigma \in \beta G_L \mid \qq \in A(\Kbar(\sigma)) \text{ and } \dim_{\mathbb{F}_\ell} W(\sigma) = d\}), \\
    R_{\qq, d} &= \sum_{h = d}^{2 g} T_{\qq, h} \\
    &= \mu(\{\sigma \in \beta G_L \mid \qq \in A(\Kbar(\sigma)) \text{ and } \dim_{\mathbb{F}_\ell} W(\sigma) \ge d\}).
\end{align*}

\begin{lemma}\label{Removingq}
For all sufficiently large prime numbers $\ell$ the following statement holds: for any $\qq \in X_\ell(\pp)$ and $0 \le d \le 2 g$, we have
\[ R_{\qq, d} = \frac{1}{\ell^{2 g}} \mu(Y_d). \]
\end{lemma}

\begin{proof}
We take $\ell$ satisfying that $[L(\qq, A[\ell]) : L(A[\ell])] = \ell^{2 g}$.
By Proposition~\ref{prop-JP-key}, all sufficiently large $\ell$ satisfy this condition.
Then the fields $L(\qq)$ and $L(A[\ell])$ are linearly disjoint over $L$.
In addition, we have $[K(\qq, A[\ell]) : K(A[\ell])] = \ell^{2 g}$ and the fields $K(\qq)$ and $K(A[\ell])$ are linearly disjoint over $K$.

We have
\begin{align*}
    &\phantom{= {}} \{\sigma \in \beta G_L \mid \qq \in A(\Kbar(\sigma)) \text{ and } \dim_{\mathbb{F}_\ell} W(\sigma) \ge d\} \\
    &= \beta G_L \cap \{\sigma \in G_K \mid \qq \in A(\Kbar(\sigma))\} \cap \{\sigma \in G_K \mid \dim_{\mathbb{F}_\ell} W(\sigma) \ge d\} \\
    &= G_{K(\qq)} \cap Y_d
\end{align*}
and see that the set $Y_d$ is the union of some left cosets of $G_{L(A[\ell])}$ in $G_K$.
Moreover, we calculate
\begin{align*}
    [L(\qq, A[\ell]) : K] &= [L(\qq, A[\ell]) : L] [L : K] \\
    &= [L(\qq) : L] [L(A[\ell]) : L] [L : K] \\
    &= [K(\qq) : K] [L(A[\ell]) : K].
\end{align*}
Therefore $K(\qq)$ and $L(A[\ell])$ are linearly disjoint over $K$.
Now we can apply~\cite[Lemma~21.3.7]{FriedJ} and
\begin{align*}
    R_{\qq, d} = \mu(G_{K(\qq)}) \mu(Y_d) = \frac{1}{\ell^{2 g}} \mu(Y_d),
\end{align*}
as desired.
\end{proof}

\begin{proof}[Proof of Proposition~{\textup{\ref{prop-ComputationOfTheMeasure}}}]
By Lemma~\ref{PartitionByWq}, we have
\[ \mu(S_\ell(\pp) \cap \beta G_L) = \sum_{W \in \mathcal{W}} \sum_{\qq \in X_\ell(\pp) / {\sim_W}} Z_{\qq, W}. \]
Since $\qq \sim_W \qq'$ implies $Z_{\qq, W} = Z_{\qq', W}$ and each coset of $X_\ell(\pp) / {\sim_W}$ has $\ell^{\dim_{\mathbb{F}_\ell} W}$ elements, we have
\begin{align*}
    \mu(S_\ell(\pp) \cap \beta G_L) &= \sum_{d = 0}^{2 g} \sum_{W \in \mathcal{W}_d} \sum_{\qq \in X_\ell(\pp)} \frac{1}{\ell^d} Z_{\qq, W} \\
    &= \sum_{\qq \in X_\ell(\pp)} \sum_{d = 0}^{2 g} \frac{1}{\ell^d} T_{\qq, d}.
\end{align*}
By the definition of $R_{\qq, d}$, we have
\[ T_{\qq, d} = \begin{cases}
R_{\qq, d} - R_{\qq, d + 1} & 0 \le d \le 2 g - 1; \\
R_{\qq, 2 g} & d = 2 g.
\end{cases}
\]
Hence we calculate
\begin{align*}
    \mu(S_\ell(\pp) \cap \beta G_L) &= \sum_{\qq \in X_\ell(\pp)} \left(\sum_{d = 0}^{2 g - 1} \frac{1}{\ell^d} (R_{\qq, d} - R_{\qq, d + 1}) + \frac{1}{\ell^{2 g}} R_{\qq, 2 g} \right) \\
    &= \sum_{\qq \in X_\ell(\pp)} \left(R_{\qq, 0} + \sum_{d = 1}^{2 g}\left(- \frac{1}{\ell^{d - 1}} + \frac{1}{\ell^d}\right)R_{\qq, d}\right) \\
    &= \sum_{\qq \in X_\ell(\pp)} R_{\qq, 0} - \sum_{d = 1}^{2 g} \sum_{\qq \in X_\ell(\pp)} \frac{\ell - 1}{\ell^d} R_{\qq, d}.
\end{align*}
By Lemma~\ref{Removingq} and $\# X_\ell(\pp) = \ell^{2 g}$, we obtain
\[ \mu(S_\ell(\pp) \cap \beta G_L) = \mu(Y_0) - \sum_{d = 1}^{2 g} \frac{\ell - 1}{\ell^d} \mu(Y_d) = \frac{1}{[L : K]} - \sum_{d = 1}^{2 g} \frac{\ell - 1}{\ell^d} \mu(Y_d), \]
which completes the proof of the proposition.
\end{proof}

\subsection{Conclusion}

In this subsection, we combine the results obtained so far and derive the AVKF-ness part of Theorem~\ref{thm-KF}.
Proposition~\ref{prop-ComputationOfTheMeasure} reduces the estimation of $\mu(S_\ell(\pp) \cap \beta G_L)$ to the estimation of $\mu(Y_d)$.
The result by Zywina~\cite{Zywina16} gives the lower bound of $\mu(Y_1)$ that holds for suitable $L$ and for $\ell$ belonging to some set of prime numbers with positive Dirichlet density.
In applying it to our proof, a slight modification is needed.
We briefly recall Zywina's result and his argument.

\begin{theorem}[Zywina~{\cite[Proposition~1.2]{Zywina16}}]\label{zywina}
Let $K$ be a number field and $A$ an abelian variety over $K$ of dimension $g \ge 1$.
Then there are a finite Galois extension $L$ of $K$, a set $\mathcal{P}$ of prime numbers of positive Dirichlet density, and a positive constant $c$ such that
\[ \frac{\mu(\{\sigma \in \beta G_L \mid A(\Kbar(\sigma)) \neq \mathbf{0}\})}{\mu(\beta G_L)} \ge \frac{c}{\ell} \]
for each $\ell \in \mathcal{P}$ and $\beta \in G_K$.
\end{theorem}

In the proof of this theorem, Zywina first takes an appropriate finite Galois extension $L$ of $K$ using the following theorem by Serre.
Then he constructs $\mathcal{P}$ and $c$ in the theorem.
Hence retaking $L$ sufficiently large for our purpose before constructing $\mathcal{P}$ and $c$ is allowed.

\begin{theorem}[Serre~{\cite{Serre86}}, cf.~{\cite[Theorem~3.1]{Zywina16}}]\label{serrecourse}
Let $K$ be a number field and $A$ an abelian variety over $K$ of dimension $g \ge 1$.
For a prime number $\ell$, let $a_\ell: G_K \to \mathrm{Aut}(A[\ell]) \cong \mathrm{GL}_{2 g}(\mathbb{F}_\ell)$ be the Galois representation on $A[\ell]$ (see Section~\textup{\ref{sec-independence}}).
Then there are a finite Galois extension $L$ of $K$ and positive integers $N$, $r$, and $\ell_0$ such that the following conditions are hold:
\begin{enumerate}[\textup{(\alph{enumi})}]
    \item For all prime numbers $\ell \ge \ell_0$, there is a connected, reductive subgroup $H_\ell$ of $\mathrm{GL}_{2 g, \mathbb{F}_\ell}$ of rank $r$ such that $a_\ell(G_L)$ is a subgroup of $H_\ell(\mathbb{F}_\ell)$ of index dividing $N$.
    Moreover, $H_\ell$ contains the group $\Gm$ of homotheties.
    \item The homomorphism $\prod_\ell a_\ell: G_L \to \prod_\ell a_\ell(G_L)$ is surjective.
\end{enumerate}
\end{theorem}

\begin{proposition}\label{prop-appliedzywina}
Let $K$ be a number field, $A$ a simple abelian variety over $K$, and $\pp$ a non-torsion point in $A(K)$.
There are a finite Galois extension $L$ of $K$, a set $\mathcal{P}$ of prime numbers with positive Dirichlet density, and a positive constant $c$ such that
\[ \mu(S_\ell(\pp) \cap \beta G_L) \le \frac{1}{[L : K]} - \frac{c}{\ell} \]
for any $\ell \in \mathcal{P}$ and any $\beta \in G_K$.
\end{proposition}

\begin{proof}
Let $L'$ be the finite Galois extension of $K$ obtained as $L$ in Theorem~\ref{serrecourse}.
We can take a finite Galois extension $L$ of $K$ containing $L'$ and satisfying the statement in Proposition~\ref{prop-independence}.
Then $L$ satisfies the conditions in Theorem~\ref{serrecourse} again after retaking $N$ larger.
By Theorem~\ref{zywina}, there are a set $\mathcal{P}$ of prime numbers with positive Dirichlet density and a positive constant $c'$ such that
\[ \mu(Y_1) \ge \frac{c'}{\ell} \]
for any $\ell \in \mathcal{P}$ and any $\beta \in G_K$.
By Proposition~\ref{prop-ComputationOfTheMeasure}, after discarding finitely many prime numbers from $\mathcal{P}$ if required, we have
\[ \mu(S_\ell(\pp) \cap \beta G_L) \le \frac{1}{[L : K]} - \frac{\ell - 1}{\ell} \mu(Y_1) \le \frac{1}{[L : K]} - \frac{\ell - 1}{\ell} \cdot \frac{c'}{\ell} \le \frac{1}{[L : K]} - \frac{c'}{2 \ell} \]
for any $\ell \in \mathcal{P}$.
Here the last inequality follows from $\ell \ge 2$.
This implies that we can take $c = c' / 2$ as a desired constant.
\end{proof}

We are now in a position to prove Lemma~\ref{AVKF-Goal} and the AVKF-ness part of Theorem~\ref{thm-KF}.

\begin{proof}[Proof of Theorem~\textup{\ref{thm-KF}} (for the AVKF-ness)]
What remains to be shown is the claim of Lemma~\ref{AVKF-Goal}.
We define $\nu = [L : K] \mu$.
Fix $\beta \in G_K$.
Then we have $\nu(\beta G_L) = 1$ and may view $(\beta G_L, \nu)$ as a probability space.
Set $U_\ell = S_\ell(\pp) \cap \beta G_L$.
By Proposition~\ref{prop-appliedzywina}, we have
\[ \nu(U_\ell) \le 1 - \frac{c [L : K]}{\ell} \]
for any $\ell$ belonging to some set $\mathcal{P}$ of prime numbers with positive Dirichlet density.
We fine that
\[ \beta G_L \to \prod_{\ell \in \mathcal{P}} a_\ell(\beta G_L) \]
is surjective and thus the family $\{U_\ell\}_{\ell \in \mathcal{P}}$ is $\nu$-independent in $\beta G_L$.
Hence we have
\[ \nu\left(\bigcap_{\ell \in \mathcal{P}; \ell \le n} U_\ell\right) = \prod_{\ell \in \mathcal{P}; \ell \le n} \nu(U_\ell) \le \prod_{\ell \in \mathcal{P}; \ell \le n} \left(1 - \frac{c [L : K]}{\ell}\right) \]
for any positive integer $n$.
Taking $n \to \infty$ implies $\nu(\bigcap_{\ell \in \mathcal{P}} U_\ell) = 0$.
Since 
\[ N(\pp) \cap \beta G_L = \bigcap_\ell U_\ell \subset \bigcap_{\ell \in \mathcal{P}} U_\ell \]
and $\nu = [L : K] \mu$, we have $\mu(N(\pp) \cap \beta G_L) = 0$.
Summing up over all cosets of $G_K / G_L$, we obtain $\mu(N(\pp)) = 0$.
\end{proof}

\section{Corollaries to Theorem~\ref{thm-KF}}\label{sec-cor}

\subsection{Existence of Kummer-faithful fields with abelian absolute Galois group}

As described in Introduction, Theorem~\ref{thm-KF} immediately implies that there exists an algebraic extension of any number field that is Kummer-faithful and has abelian absolute Galois group.
Here we prove that such fields exist in abundance.

\begin{corollary}\label{cor-continuummany}
Let $K$ be a number field.
There are continuum many algebraic extensions of $K$ up to $K$-isomorphism which are Kummer-faithful fields with abelian absolute Galois group.
\end{corollary}

\begin{proof}
It is known that, if $B$ is a measurable subset of $G_K$ of positive measure, then the set $\{\Kbar(\sigma) \mid \sigma \in B\}$ contains continuum many algebraic extensions of $K$ up to $K$-isomorphism~\cite[Theorem~7.1]{Jarden74}.
Set $B$ to be the set of $\sigma \in G_K$ satisfying that the field $\Kbar(\sigma)$ is Kummer-faithful.
Then the measure of $B$ is one by Theorem~\ref{thm-KF} and the absolute Galois group of $\Kbar(\sigma)$ is abelian for each $\sigma \in G_K$.
Therefore the corollary follows.
\end{proof}

It is proved that $G_{\Kbar(\sigma)}$ is isomorphic to $\widehat{\mathbb{Z}}$ for almost all $\sigma \in G_K$~\cite[Theorem~5.1]{Jarden74} (for number fields, it was originally proved by Ax~\cite[Proposition~3]{Ax}).
Murotani~\cite[Theorem~B~(ii)]{Murotani} recently proved that, for an algebraic extension $K$ of a finite field, $K$ is Kummer-faithful if and only if $G_K \cong \widehat{\mathbb{Z}}$.
These facts naturally raise the following question.

\begin{question}
Let $K$ be a number field.
Describe the relation between the two conditions for $\sigma \in G_K$: that $\Kbar(\sigma)$ is Kummer-faithful and that $G_{\Kbar(\sigma)} \cong \widehat{\mathbb{Z}}$.
\end{question}

\subsection{Application to the composite of Kummer-faithful fields}

Ozeki and Taguchi~\cite{OT} discussed whether the composite field of two Kummer-faithful Galois extensions of a perfect field is again a Kummer-faithful field, and provided a counterexample~\cite[Example~3.7]{OT} in the $p$-adic number field case.
In the global number field case, we show that the composite field of two Kummer-faithful non-Galois extensions of a number field is not always Kummer-faithful; moreover, it can be the algebraic closure.

\begin{corollary}
Let $K$ be a number field.
There exist two Kummer-faithful extensions $F_1$ and $F_2$ of $K$ such that $F_1 F_2 = \Kbar$.
In particular, $F_1 F_2$ is not Kummer-faithful.
\end{corollary}

\begin{proof}
Fix two positive integers $e_1$ and $e_2$.
Let
\[ B_i = \{\sigma \in G_K^{e_i} \mid \Kbar(\sigma) \text{ is Kummer-faithful}\} \]
for $i = 1, 2$ and let
\[ C = \{(\sigma_1, \sigma_2) \in G_K^{e_1} \times G_K^{e_2} \mid \Kbar(\sigma_1) \Kbar(\sigma_2) = \Kbar \}. \]
If $(\sigma_1, \sigma_2) \in G_K^{e_1} \times G_K^{e_2}$ belongs to the intersection $(B_1 \times B_2) \cap C$, then $F_1 = \Kbar(\sigma_1)$ and $F_2 = \Kbar(\sigma_2)$ satisfy the statement in the corollary.
Thus it suffices to show that $(B_1 \times B_2) \cap C$ is nonempty.
By Theorem~\ref{thm-KF}, each $B_i$ has measure one in $G_K^{e_i}$ and $B_1 \times B_2$ has measure one in $G_K^{e_1} \times G_K^{e_2}$.
By~\cite[Theorem~5.1]{Jarden74}, the set $C$ has measure one in $G_K^{e_1} \times G_K^{e_2}$.
Hence $(B_1 \times B_2) \cap C$ has also measure one and is thus nonempty.
\end{proof}

Note that the field $\Kbar(\sigma)$ is far from being Galois: in fact, $\Kbar(\sigma)$ is a Galois extension of no proper subfield for almost all $\sigma \in G_K^e$~\cite[Theorem~7.10]{BS}.
It is still open that the composite field of two Kummer-faithful Galois extensions of a number field is again a Kummer-faithful field.

\bibliographystyle{amsalpha}

\end{document}